\documentclass{amsart}
\usepackage{amsmath,amsfonts,amsthm,amscd,amssymb,pstricks,graphicx,subfig,sidecap}

\newtheorem{thm}{Theorem}[section]

\newtheorem{cor}[thm]{Corollary}
\newtheorem{lem}[thm]{Lemma}

\newtheorem{defn}[thm]{Definition}
\newtheorem{question}[thm]{Question}

\begin{document}

\author{Brendan Murphy and Jonathan Pakianathan }
\title{Kakeya configurations in Lie groups and Homogeneous spaces.}

\begin{abstract}
In this paper, we study continuous Kakeya line and needle configurations, of both the oriented and unoriented varieties, in connected Lie groups 
and some associated homogenous spaces. These are the analogs of Kakeya line (needle) sets (subsets of $\mathbb{R}^n$ where it is 
possible to turn a line (respectively an interval of unit length) through all directions {\bf continuously, without repeating a "direction"}.) 
We show under some general assumptions 
that any such continuous Kakeya line configuration set in a connected Lie group must contain an open neighborhood of the identity, 
and hence must have positive Haar measure. In connected nilpotent Lie groups $G$, the only subspace 
of $G$ that contains such an unoriented line configuration is shown to be $G$ itself.  Finally some similar questions in homogeneous spaces are addressed.

These questions were motivated by work of Z. Dvir in the finite field setting.


\end{abstract}

\maketitle

\tableofcontents

\section{Introduction}

The purpose of this note is to study Kakeya like configurations in the setting of connected Lie groups. We will only be looking at continuous configurations and so this paper has little to say about the famous Kakeya conjecture (described below) but instead focuses on Kakeya needle-like configurations in general settings. 

Kakeya posed the Kakeya needle question in $\mathbb{R}^2$: What is the smallest area of a subset of $\mathbb{R}^2$ for which it is possible to turn a 
needle (interval of length 1) around {\bf continuously} while remaining entirely in the subset. Such a subset is called a Kakeya needle set. When the subset is required to be convex, Pal (\cite{Pa}) showed the answer was given by the equilateral triangle of base height 1 (area $\frac{1}{\sqrt{3}}$). 
Similar minimizers of positive area exist when the subset is required to be star-convex. However Besicovitch (\cite{Be1}, \cite{Be2}) showed that in general, 
there is no lower bound on the area, there exist Kakeya needle subsets of the plane with arbitrarily small positive Lebesgue measure. 
Cunningham (\cite{Cu}) later refined this work to show there also exist simply-connected Kakeya needle subsets of the unit disk of arbitrarily small positive Lebesgue measure.

Besicovitch also considered more general Kakeya sets which were defined as subsets of $\mathbb{R}^n$ which contain a line segment of length 1 
in every given direction of $\mathbb{R}^n$. Here, unlike the case for Kakeya needle sets, the variation of the needle as a function of direction need not 
be continuous, or even Borel measurable. Besicovitch showed that there are Kakeya sets of Lebesgue measure zero in every $\mathbb{R}^n, n \geq 2$. 
Regardless these examples have maximal Hausdorff dimension $n$ and the famous Kakeya conjecture asks to show that this is true in general, i.e., 
that every Kakeya subset of 
$\mathbb{R}^n$ has Hausdorff (and hence Minkowski) dimension $n$.  This conjecture is still open for $n \geq 3$. This conjecture is related to many 
questions in Harmonic analysis, for example C. Fefferman (\cite{Fe}) used Besicovitch's construction of Kakeya sets of measure zero to disprove the ball multiplier 
conjecture for $L^p(\mathbb{R}^n), p \neq 2$. Further interesting applications to analysis are discussed in \cite{Bo}, \cite{Ta} and \cite{Wo}.

In this paper we will not consider Kakeya sets but only configurations similar to the Kakeya needle sets where the needle is required to move continuously. In fact, we will most often be interested in continuous line configurations where the motion goes thru a line in each direction (or parallel class) exactly once. These are sets in which it is possible to move a line continuously such that every direction in the ambient space is achieved exactly once. 

More precisely, we will consider continuous functions
$$
\sigma: \mathbb{R}P^{n-1} \to \mathbb{R}^n
$$
where $\mathbb{R}P^{n-1}$ is the projective space of $\mathbb{R}^n$, i.e., the space of lines through the origin in $\mathbb{R}^n$. 

For each line $L$ through the origin, $\sigma$ picks out a displacement $\sigma(L)$ in a continuous manner and we study the set
$ |\sigma|=\cup_{L \in \mathbb{R}P^{n-1}} (\sigma(L) + L) $.
Notice this set contains a line in every direction but is more constrained as the line placement varies continuously and since the configuration is minimal 
in the sense that for each line through the origin, {\bf exactly one} line parallel to it is used during the motion.
We call such a $\sigma$, a continuous unoriented Kakeya line configuration and $|\sigma|$ is the underlying space of the motion.

Similarly we consider continuous functions
$$
\sigma: S^{n-1} \to \mathbb{R}^n
$$
and call them continuous oriented Kakeya line configurations.

Notice for either the oriented or unoriented configurations above, one can consider either the space $|\sigma|$ swept out of the lines of the configuration or 
the smaller needle space $|\hat{\sigma}| = \cup_{L \in \mathbb{R}P^{n-1}} (\sigma(L) + I(L))$ where $I(L)$ is the interval of length $1$ centered about the origin 
in $L$. The classical Kakeya conjecture is equivalent to showing for any $\sigma: \mathbb{R}P^{n-1} \to \mathbb{R}$ (not necessarily continuous or even Borel measurable), that any Borel set containing $|\hat{\sigma}|$ has Hausdorff dimension $n$. In this paper we will not address this and consider exclusively 
only continuous $\sigma$.

This study was motivated by a paper of Dvir (\cite{Dv}) which looked at an analogous question over finite fields. Dvir showed 
that if $E$ is a subset of a finite dimensional vector space over a finite field which contains a line in every direction, then $E$ contains a "positive proportion" of 
the vector space. Under suitable assumptions, we recover the same picture for continuous Kakeya line configurations in this paper.

We define an analog of this question in any connected Lie group and in fact any homogeneous space made from such a Lie group and 
provide partial answers to the basic questions about these Kakeya configurations. Here the projective space (respectively sphere) of the Lie algebra $\mathfrak{g}$ 
of the Lie group $G$ becomes the space of unoriented ( respectively oriented) "directions" and one-parameter subgroups take the place of lines. \\
(Though one could generalize even further to Riemannian manifolds and geodesics, to get a single controlling "direction" space, one would have to assume 
a very controlled parallelization of the tangent bundle and we do not pursue this here.)

The techniques are mostly from algebraic topology and though some more can be said using measure theory and harmonic analysis, we 
will not do so in this note, choosing to focus primarily on the topological aspects. As a course, we will not shy away from using some basic machinery, 
even though a more elementary proof might be available, if we feel the general viewpoint sheds more light on the situation in other spaces.

The underlying space of the configuration $\sigma: \mathbb{R}P^{n-1} \to \mathbb{R}^n$ is defined to be 
$|\sigma| = \cup_{L \in RP^{n-1}} (\sigma(L) + L)$. A continuous unoriented Kakeya line set is a subset of $\mathbb{R}^n$ that contains the underlying space 
of some continuous unoriented Kakeya line configuration. Intuitively, it is a set where it is possible to move a line continuously such that during the motion 
every parallel class of lines is used exactly once.

Using some basic algebraic topology, we show:

\begin{thm}
The only continuous unoriented Kakeya line set in $\mathbb{R}^n$ is $\mathbb{R}^n$ itself, i.e., for every continuous map $\sigma: \mathbb{R}P^{n-1} 
\to \mathbb{R}^n$, $|\sigma| = \cup_{L \in \mathbb{RP}^{n-1}} (\sigma(L) + L) = \mathbb{R}^n$. 
\end{thm}

Examples are also given showing that this theorem fails if either the continuity assumption is dropped or if more than one line from each parallel class 
is allowed during the motion (i.e., if $\sigma$ assigns more than one output to an input and is not a function). 

In the oriented case we show:

\begin{thm}
The only continuous oriented Kakeya line set  in $\mathbb{R}^{2n+1}$ is $\mathbb{R}^{2n+1}$ itself, i.e., for every continuous map $\sigma: S^{2n} 
\to \mathbb{R}^{2n+1}$, $|\sigma| = \cup_{x \in S^{2n}} (\sigma(x) + L_x) = \mathbb{R}^{2n+1}$ for any $n \geq 1$. 
Here $L_x$ is the line passing thru the origin and $x$ oriented in direction pointing from the origin to $x$. 
On the other hand, in $\mathbb{R}^{2n}$ there exist continuous oriented Kakeya line sets $|\sigma|$ that are disjoint from any prescribed bounded set, and so in particular 
$|\sigma|$ need not be all of $\mathbb{R}^{2n}$. 
\end{thm}

Again this theorem fails if either $\sigma$ is not continuous or if $\sigma$ is not a function, i.e., more than one line in a given direction is used during 
the motion.

It is now not hard to extend these questions to the setting of connected Lie groups $G$. Such a Lie group has a Lie algebra $\mathfrak{g}$, 
and it is well known (see Warner) that the Lie subalgebras of $\mathfrak{g}$ are in one to one correspondence with the connected subgroups of 
$G$. In particular for every line through the origin in $\mathfrak{g}$ there corresponds a unique one-parameter subgroup of $G$ (which will either be an 
immersed line or a circle group.) through the exponential-log correspondence. The one-parameter subgroups of $\mathbb{R}^n$ are exactly the lines through the origin and these will take the place of lines in a general connected Lie group. The analog of parallel lines, will be left translates $gH$ of a given 
one parameter subgroup $H$. (One can also replace left with right thru out this paper without changing the results though the translates in general 
will be different in a non-Abelian Lie group).

Let $P(\mathfrak{g})$ be the projective space corresponding to the finite dimensional real vector space $\mathfrak{g}$. For a line $L \in \mathfrak{g}$ 
we let $\exp(L)$ denote the corresponding one-parameter subgroup of $G$ and $g \exp(L)$ for the typical "parallel" left translate. We then define:

\begin{defn} Let $G$ be a connected Lie group with Lie algebra $\mathfrak{g}$, and let $P(\mathfrak{g})$ be the projective space 
associated to $\mathfrak{g}$. A continuous unoriented Kakeya line configuration is a continuous map $\sigma: P(\mathfrak{g}) \to G$. 
The underlying space to this configuration, denoted $|\sigma|$ is defined as $\sigma= \cup_{L \in P(\mathfrak{g})}(\sigma(L)\exp(L))$. 
A continuous unoriented Kakeya line set $X$ is a subspace of $G$ which contains $|\sigma|$ for some continuous unoriented Kakeya line configuration $\sigma$.
Similar definitions hold for oriented Kakeya line configurations where $P(\mathfrak{g})$ is replaced by $S(\mathfrak{g})$, the sphere of $\mathfrak{g}$ 
with respect to some positive definite inner product on $\mathfrak{g}$.
\end{defn}

We prove, among other things:

\begin{thm} Let $G$ be a connected nilpotent Lie group, then the only continuous unoriented Kakeya line set in $G$ is $G$ itself, i.e., 
$|\sigma|=G$ for any continuous unoriented Kakeya configuration $\sigma$ in $G$.
A similar theorem holds for oriented Kakeya line sets if in addition, $G$ is assumed to be odd dimensional. This theorem 
also holds for connected solvable Lie groups of type E.
\end{thm}

Again, the theorem is easily seen to not be true without the continuity assumption. 

In a general connected Lie group, one can prove weaker results. One has to assume the Kakeya line configuration is "linear" (liftable via the exponential map) or taut (see sections below) 
which follows for example if $\sigma: P(\mathfrak{g}) \to G$ has image contained in a suitable open neighborhood of the identity (one which is diffeomorphic 
to an open ball in $\mathfrak{g}$ under the exponential-log correspondence.) Then one has:

\begin{thm} Let $G$ be a connected Lie group and let $\sigma: P(\mathfrak{g}) \to G$ be a linear continuous unoriented Kakeya line configuration, 
then $|\sigma|$ contains the identity element. Furthermore any taut unoriented Kakeya line configuration $\sigma$ has $|\sigma|$ contain an open neighborhood of the identity in $G$ and hence has positive (left) Haar measure. The same statements hold for oriented Kakeya line configurations if the dimension of $G$ is odd.
\end{thm}

Finally the definition of continuous Kakeya line configuration $\sigma: P(\mathfrak{g}) \to G$ behaves well with respect to Lie group covering maps 
$\pi: G \to M$. Here $\mathfrak{g}=\mathfrak{M}$ canonically and $\pi \circ \sigma$ is then a continuous Kakeya line configuration in $M$. 
Furthermore, basic covering space theory shows (see section below) we can lift continuous Kakeya line configurations from $M$ to $G$ under 
mild hypothesis. Thus for a given finite dimensional, real, Lie algebra $\mathfrak{g}$ it is often sufficient to understand the questions about continuous Kakeya line 
configurations in the unique simply-connected Lie group with Lie algebra $\mathfrak{g}$. 

Given a closed subgroup $K$ (whether normal or not) of a Lie group $G$, $K \backslash G$ is a smooth manifold, and one has a quotient map $\pi: G \to K\backslash G$ which takes left cosets of one-parameter subgroups of $G$ to a 
distinguished set of curves in $K\backslash G$. Thus a continuous Kakeya line configuration $\sigma: P(\mathfrak{g}) \to G$ gives rise to a continuous 
unoriented Kakeya line configuration $\hat{\sigma}=\pi \circ \sigma: P(\mathfrak{g}) \to K\backslash G$. Furthermore, given a continuous map 
$P(\mathfrak{g}) \to K\backslash G$, one can use fiber bundle theory to decide when it comes from a Kakeya configuration in $G$. We study such configurations briefly in section~\ref{section: homogeneous}.

As a final comment in this introductory section, we remark that given a (not necessarily continuous) configuration $\sigma: P(\mathfrak{g}) \to G$ in a connected Lie group, and a choice of positive definite inner product on $\mathfrak{g}$, we can form, 
$$
|\hat{\sigma}|=\cup_{L \in P(\mathfrak{g})} (\sigma(L) exp(I(L)))
$$ where $I(L)$ is the part of the line $L$ within distance $\frac{1}{2}$ of the origin. 
A Borel set containing such a $\sigma$ is the analog of a Kakeya set in a connected Lie group. 
To distinguish this from the case where $\sigma$ is continuous we will refer to these as Kakeya-Besicovitch sets.
Note that Hausdorff and Minkowski dimension are metric dependent and one can easily find examples of two metrics that metrize the same 
topological group with different dimension theory. (The $2$-adic integers with the $2$-adic metric have Hausdorff dimension 1 while with the 
metric induced from the real line by the homeomorphism with the Cantor set, it has a fractional Hausdorff dimension.)
However for any two positive definite inner products on $\mathfrak{g}$, the corresponding left-invariant Riemannian metrics on $G$ have the same 
Hausdorff dimensions for any Borel subset. We will use such a canonical version of Hausdorff dimension on a Lie group to state:

\begin{question}[Lie variant of Kakeya Conjecture]
Let $G$ be a connected Lie group and $K$ a Kakeya-Besicovitch subset of $G$. Does the Hausdorff dimension of $K$ have to equal the dimension of $G$?
\end{question}

Note the case of simply-connected, Abelian Lie groups, i.e., $(\mathbb{R}^n, +)$ is exactly the classical Kakeya conjecture which is known in dimensions 1 and 2, 
(see \cite{Da}).
The Lie group $\{ \begin{bmatrix} a & b \\ 0 & 1 \end{bmatrix} | a > 0, b \in \mathbb{R} \}$ is the smallest nonAbelian example to consider where an answer might 
be obtainable from current technology. Though the underlying space of this example is diffeomorphic to $\mathbb{R}^2$, the one-parameter subgroups do 
not correspond to lines in general under this correspondence.

Also we state a discrete version of this question, let us define a subset $E$ of a finite group $G$ to be a Kakeya set if it contains a left coset of every cyclic subgroup of $G$. 
We then ask: 

\begin{question}[Discrete variant of Kakeya Conjecture]
Is there a "sharp" explicit formula for a positive constant $c=c(e,g)$, such that for every finite $g$-generated group $G$ of 
exponent dividing $e$, and Kakeya subset $E \subseteq G$, one has $|E| \geq c|G|$.
\label{question: discrete}
\end{question}

Note, by the solution of the restricted Burnside conjecture (\cite{Ze}, \cite{Ze2}), there are only finitely many $g$-generated groups whose exponent divides $e$. 
It is clear in this case there is an absolute constant $c=c(e,g) > 0$ such that the size of any Kakeya set in these groups is at least $c|G|$. 
(This is because in any given group, a Kakeya set has to have at least size $1$ and so has to be at least $\frac{1}{|G|}$ worth of $G$. Thus 
we may take $c(e,g)$ to be the reciprocal of the size of the largest $g$-generated subgroup of exponent $e$. This $c$ works but is somewhat silly 
and definitely not sharp.) 
Thus the existence of a positive $c=c(e,g)$ in the last question is clear, but a sharp, useful, explicit formula for it would be nice. 

In work of Dvir (\cite{Dv}) on the case of $G$ a finite product of cyclic groups of prime order, a relatively sharp formula for $c$ was found as were examples showing 
that it must depend on $g$, the number of generators of the group. Sharper formulas for $c$ in this case were obtained in 
\cite{SS}. In \cite{DH}, the case of $G$ a finite product of two cyclic groups of order $p^k$ 
was worked out and an explicit sharp formula for $c$ found. In this work, examples were found showing $c$ must depend on $e$, the exponent of $G$. 

Thus the question as framed definitely has a solution and the dependence of the constant on $g$ and $e$ as stated is necessary - explicit sharp formulas for $c=c(e,g) > 0$ 
are the hopeful goal.

\section{Continuous Kakeya Line Configurations in $\mathbb{R}^n$}

\begin{defn} A continuous unoriented Kakeya line configuration in $\mathbb{R}^n$ is a continuous map 
$\sigma: \mathbb{R}P^{n-1} \to \mathbb{R}^n$. The underlying space of $\sigma$, $|\sigma|$ is defined as
$$
|\sigma| = \cup_{L \in \mathbb{R}P^{n-1}} (\sigma(L) + L)
$$
The associated unoriented Kakeya needle set is $|\hat{\sigma}| = \cup_{L \in \mathbb{R}P^{n-1}}(\sigma(L) + I(L))$ 
where $I(L)$ is the interval of length 1 centered around the origin in $L$. For $1 \leq R < \infty$, the $R$-elongation of the Kakeya needle set 
is the set obtained when we replace $I(L)$ with an interval of length $R$ centered abound the origin in $L$ in this definition. 
Similar definitions hold for oriented variants where $\mathbb{R}P^{n-1}$ is replaced by $S^{n-1}$.
\end{defn}

Note since projective spaces, spheres and intervals are compact spaces, it is not hard to show that the associated Kakeya needle sets (oriented or unoriented) 
and their elongations are compact subspaces of $\mathbb{R}^n$ when $\sigma$ is continuous. Thus $|\sigma|$ itself is $\sigma$-compact and 
hence a Borel set.

Consider the trivial vector bundle $\epsilon^n$ given by $\pi: \mathbb{R}P^{n-1} \times \mathbb{R}^n \to \mathbb{R}P^{n-1}$. 
Notice, for any $L \in \mathbb{R}P^{n-1}$,  the fiber $\pi^{-1}(L)$ splits as $\mathbb{R}^n = L \oplus L^{\perp}$, where $L^{\perp}$ is the orthogonal complement of $L$ 
with respect to the standard dot product of $\mathbb{R}^n$.  Thus $\epsilon^n = \gamma \oplus \gamma^{\perp}$ where 
$\gamma$ is the canonical line bundle of $\mathbb{R}P^{n-1}$ and $\oplus$ is Whitney sum of vector bundles (see \cite{MS}). 

We will need the well known fact that any continuous section of $\gamma^{\perp}$ has a zero. 

\begin{lem}
Every continuous section of $\gamma^{\perp}$ has a zero.
\label{lem: zero}
\end{lem}
\begin{proof}
We recall a proof here quickly.
First we recall some basic facts, see \cite{MS} for details: any vector bundle $\xi$ over a paracompact space has a sequence of Stiefel-Whitney classes $w_0, w_1, w_2, \dots$ 
such that $w_k(\xi) = 0$ when $k > dim(\xi)$. The class $w_i$ is an element of $H^i(X;\mathbb{F}_2)$ and they fit together to give a 
total Stiefel-Whitney class $w(\xi) = 1 + w_1(\xi) + w_2(\xi) + \dots$.  If $\xi = \xi_1 \oplus \xi_2$ then $w(\xi)=w(\xi_1)w(\xi_2)$ by Cartan's formula.
Finally recall $w(\epsilon^k)=1$ for any trivial bundle $\epsilon^k$ and $w(\gamma)=1+a$ where $a$ is a generator of $H^1(\mathbb{R}P^{n-1}, \mathbb{F}_2)$. 

Since $\epsilon^n = \gamma \oplus \gamma^{\perp}$ we have by applying $w$ to both sides and using Cartan's formula that
$$
1 = w(\gamma)w(\gamma^{\perp}) = (1+a)w(\gamma^{\perp})
$$
from which it follows $w(\gamma^{\perp}) = \frac{1}{1+a}=1+a+a^2+ \dots + a^{n-1}$ as we are in characteristic two and 
$H^*(\mathbb{R}P^{n-1};\mathbb{F}_2) = \mathbb{F}_2[a]/(a^n)$ is a truncated polynomial algebra. 

In particular $w_{n-1}(\gamma^{\perp}) = a^{n-1} \neq 0$. Now suppose that there existed a nowhere vanishing section of $\gamma^{\perp}$. Then 
$\gamma^{\perp} = \mu \oplus \epsilon^1$ where $\mu$ is a $(n-2)$-dimensional vector bundle. Thus Cartan's formula would show 
that $\gamma^{\perp}$ would have the same Stiefel-Whitney classes as a $(n-2)$-dimensional vector bundle and in particular, $w_{n-1}(\gamma^{\perp})=0$ 
contradicting what we have established. Thus we conclude $\gamma^{\perp}$ possesses no nowhere vanishing continuous section. 
\end{proof}
 
\begin{lem}
Let $\sigma: \mathbb{R}P^{n-1} \to \mathbb{R}^n$ be a continuous unoriented Kakeya line configuration. Then 
$\hat{0} \in |\sigma|$.
\label{lem: containzero}
\end{lem}
\begin{proof}
For $L \in \mathbb{R}P^{n-1}$, let $P_L : \mathbb{R}^n \to L^{\perp}$ denote the orthogonal projection onto $L^{\perp}$.
It is not hard to check that the function $L \to P_L(\sigma(L))$ is a continuous section of $\gamma^{\perp}$ 
and that $|\sigma| = \cup_{L \in \mathbb{R}P^{n-1}} (\sigma(L) + L) = \cup_{L \in \mathbb{R}P^{n-1}}(P_L(\sigma(L)) + L)$.
Since $P_L(\sigma(L))$ must have a zero by lemma~\ref{lem: zero}, we conclude that $|\sigma|$ must contain a line $L$ through the origin 
and hence must contain the origin $\hat{0}$.
\end{proof}

\begin{thm}
Let $\sigma: \mathbb{R}P^{n-1} \to \mathbb{R}^n$ be a continuous unoriented Kakeya line configuration, then 
$|\sigma| = \mathbb{R}^n$.
\label{thm: Rn-unoriented}
\end{thm}
\begin{proof}
For any $\hat{x} \in \mathbb{R}^n$, let $T_{-\hat{x}}$ denote the operator of translation by $-\hat{x}$. Note 
that $T_{-\hat{x}} \circ \sigma$ is another continuous unoriented Kakeya line configuration and 
$|T_{-\hat{x}} \circ \sigma| = |\sigma| - \hat{x}$. 
By lemma~\ref{lem: containzero}, this set must contain the origin. Thus $|\sigma|$ itself must contain $\hat{x}$. 
Since $\hat{x}$ was arbitrary, $|\sigma| = \mathbb{R}^n$.
\end{proof}

\begin{cor}
Let $\sigma$ be a continuous unoriented Kakeya line configuration, then for some $1 \leq R < \infty$, the $R$-elongation of the corresponding 
unoriented Kakeya needle set has positive Lebesgue measure. 
\end{cor}
\begin{proof}
$|\sigma |=\mathbb{R}^n$ is the nested union of the $N$-elongations of the corresponding Kakeya needle set, where $N \in \mathbb{Z}_+$.
The result follows immediately.
\end{proof}

We now shift gears and consider oriented configurations $\sigma: S^{n-1} \to \mathbb{R}^n$. 
Notice that the trivial vector bundle $\epsilon^n$ given by $\pi: S^{n-1} \times \mathbb{R}^n \to S^{n-1}$ 
has each fiber over a point $x \in S^{n-1}$ split into the normal line at $x$ and the tangent plane of the sphere at $x$. 
Since the normal bundle of the embedding of $S^{n-1}$ in $\mathbb{R}^n$ is trivial this gives 
$\epsilon^n = \epsilon^1 \oplus \tau$ where $\tau$ is the tangent bundle of the sphere $S^{n-1}$.
We may once again consider the orthogonal projection $P: \mathbb{R}^n \to T_x(S^{n-1})$ and note that 
the function $x \to P(\sigma(x))$ is a continuous section of the tangent bundle of the sphere, i.e., a continuous vector field on the sphere.

In this case the Hairy Ball Theorem says that every continuous vector field on an even dimensional sphere, has a zero (whereas the odd dimensional spheres 
have non vanishing continuous vector fields). We will not prove this theorem as the proof is found in every first year algebraic topology textbook (see \cite{Mu2}, \cite{Bre}, \cite{Ha}) and is well-known in general. Using this we prove the oriented versions of our theorems:

\begin{thm} Let $\sigma: S^{2n} \to \mathbb{R}^{2n+1}$ be a continuous oriented Kakeya line configuration. 
Then $|\sigma| = \mathbb{R}^{2n+1}$. This theorem fails in $\mathbb{R}^{2n}$ even for continuous sections. 
\label{thm: Rn-oriented}
\end{thm}
\begin{proof}
Let $p \circ \sigma$ denote the corresponding continuous section of the tangent bundle of $S^{2n}$. By the Hairy Ball Theorem, 
$p \circ \sigma$ has a zero. In other words, there exists a point $x \in S^{2n}$ such that $\sigma(x)$ is normal to the sphere and hence 
lies in the line that $x$ generates. Thus $|\sigma| = \cup_{y \in S^{2n}} (\sigma(y) + L_y)$ has $L_x$, the line through $x$ and the origin contained in 
$|\sigma|$.  Thus $\sigma$ contains the origin. From here the same trick with translations used in the unoriented case, shows that 
$|\sigma|= \mathbb{R}^{2n+1}$. 

To show that this fails for $\mathbb{R}^{2n}$ i.e., odd spheres, even with continuity, take any nowhere vanishing vector field 
on $S^{2n-1}$. For any $C > 0$, one can scale the vector field so that its norm at any point is greater than or equal to $C$.
Then it is not hard to see that $|\sigma|$ is a union of lines each with distance $\geq C$ from the origin. Thus $|\sigma|$ is disjoint 
from the open ball of radius $C$ about the origin. Since $C$ is arbitrary, we can find examples of continuous oriented Kakeya line configurations where $|\sigma|$ is disjoint from any prescribed bounded 
set in $\mathbb{R}^{2n}$. 

For an explicit example, the reader should consider the collection of (counter clockwise oriented) tangent lines to a standard circle of radius $C$ about the origin in $\mathbb{R}^2$. 
This collection has $|\sigma|$ equal to the complement of the open disk of radius $C$ and hence $|\sigma| \neq \mathbb{R}^2$. As we have shown, these examples only exist in even Euclidean dimensions and in the case of oriented configurations. This example does not count in the unoriented case 
as it has two parallel lines in each parallel class and so there is no function $\sigma: \mathbb{R}P^{1} \to \mathbb{R}^2$ corresponding to this motion 
in the unoriented case. 

\end{proof}

It follows from this theorem, that any continuous oriented Kakeya needle set in $\mathbb{R}^{2n+1}$ has an $R$-elongation with positive Lebesgue measure just like in the unoriented case.

\section{Continuous Kakeya Line Configurations in connected Lie groups}

Let $G$ be a connected Lie group and $\mathfrak{g}$ its Lie algebra. Recall that the Lie subalgebras of $\mathfrak{g}$ are in bijective correspondence 
with the connected subgroups of $G$. (See \cite{Wa}). The $1$-dimensional subspaces of $\mathfrak{g}$ are automatically Lie subalgebras and 
are in bijective correspondence (exponential-logarithm correspondence) with the $1$-parameter subgroups $H$ of $G$. In the case of $G=(\mathbb{R}^n, +)$, these are exactly the lines 
through the origin. For any $x \in G$, $xH$ is a left coset of $H$ which will serve as a "parallel" copy of $H$ going through the point $x$.

Thus in the setting of connected Lie groups, the left cosets of $1$-parameter subgroups will play the role that lines did in $\mathbb{R}^n$. 
$P(\mathfrak{g})$, the projective space of the vector space $\mathfrak{g}$, then serves as the space of "parallel-classes" i.e., the space 
of $1$-parameter subgroups of $G$. For $L \in P(\mathfrak{g})$, we will denote the corresponding $1$-parameter group by $\exp(L)$. These constructions have many analogies with the $\mathbb{R}^n$-case (and indeed reduce to it 
when $G=(\mathbb{R}^n,+))$: \\
(1) Through any point $x$ in $G$, and $L \in P(\mathfrak{g})$, there is a unique left coset ("parallel line") $x \exp(L)$ which goes through $x$. \\
(2) "Parallel lines" do not meet. \\
(3) For each $L \in \mathfrak{g}$, there is a smooth surjective homomorphism $$(\mathbb{R},+) \to \exp(L) \subseteq G.$$ However the reader is warned 
that even though this map is an immersion, the image can be a circle group (map need not be injective) and the image does not have to be closed in $G$ (for example a dense line in a torus). 

Similarly fixing a positive definite inner product on $\mathfrak{g}$, the sphere $S(\mathfrak{g})$ functions as the space of oriented one-parameter 
subgroups of $G$, where the orientation is induced by moving outward from $0$ to the point $x \in S(\mathfrak{g})$ in the corresponding one-parameter 
subgroup of $G$. 

We now make definitions that generalize those in $(\mathbb{R}^n, +)$.

\begin{defn} Let $G$ be a connected Lie group and $\mathfrak{g}$ be its Lie algebra. Fix a positive definite inner product on $\mathfrak{g}$ also. A continuous unoriented Kakeya line configuration in $G$ 
is a continuous map $\sigma: P(\mathfrak{g}) \to G$. The underlying space of $\sigma$, $|\sigma|$ is defined as
$$
|\sigma| = \cup_{L \in P(\mathfrak{g})} (\sigma(L) \star \exp(L))
$$
where $\star$ denotes multiplication in $G$. The associated unoriented Kakeya needle set is 
$|\hat{\sigma}| = \cup_{L \in P(\mathfrak{g})} (\sigma(L) \star \exp(I(L)))$ where $I(L)$ is the interval of length $1$ centered around the origin in $L$. 
For $1 \leq R < \infty$, the $R$-elongation of the Kakeya needle set is the set obtained when $I(L)$ is replaced with an interval of length $R$ 
centered about the origin in $L$, in the last definition.

We define similar oriented versions of these concepts where a continuous oriented Kakeya line configuration in $G$ is given by a continuous map 
$\sigma: S(\mathfrak{g}) \to G$, where $S(\mathfrak{g})$ is the unit-sphere in $\mathfrak{g}$. 
\end{defn}

A difficulty one first encounters in a general Lie group is that the underlying space $|\sigma|$ of a configuration $\sigma$ lives in $G$ while the direction space 
$P(\mathfrak{g})$ lives in the Lie algebra $\mathfrak{g}$. To help deal with this we make use of the exponential map 
$\exp: \mathfrak{g} \to G$ and make the following definition:

\begin{defn} A continuous unoriented Kakeya line configuration $\sigma: P(\mathfrak{g}) \to G$ is called {\bf linear} if there exists a 
continuous lift $\mu: P(\mathfrak{g}) \to \mathfrak{g}$ such that $\exp \circ \mu = \sigma$. 
We make a similar definition in the oriented case where $S(\mathfrak{g})$ replaces $P(\mathfrak{g})$. 
\end{defn}

The exponential map is a smooth map from $\mathfrak{g}$ to $G$ but it need not be surjective in general even for connected Lie groups. 
(For example for the connected Lie group $G=SL(2,\mathbb{R})$ the exponential map is not surjective. ) Thus not every continuous Kakeya line 
configuration in $G$ will be linear, for example if the image of $\sigma$ does not lie in the image of the exponential map, there can be no linearizing lift. 

However as the derivative of the exponential map at $0 \in \mathfrak{g}$ 
is the identity map $\mathfrak{g} \to \mathfrak{g}$, the exponential induces a diffeomorphism from an open neighborhood of $0$ in $\mathfrak{g}$ 
to an open neighborhood of the identity element $e$ in $G$.  We make the following definition: 

\begin{defn} A continuous unoriented Kakeya line configuration $\sigma: P(\mathfrak{g}) \to G$ is called {\bf taut} if 
the image of $\sigma$ lies in an open neighborhood of $e$ in $G$ which is diffeomorphic to an open neighborhood of $0$ in $\mathfrak{g}$ 
via (the local inverse of) the exponential map. We make similar definitions in the oriented case. Note such configurations are automatically linear. 
\end{defn}

Thus every connected Lie group has a lot of taut configurations though in general the inclusions
$$
\text{ Taut configurations} \subseteq \text{ Linear configurations} \subseteq \text{ Continuous configurations }
$$
are all proper. In $(\mathbb{R}^n,+)$ these 3 notions coincide and we will see that in connected nilpotent Lie groups of dimension $> 2$, 
that every continuous configuration is linear. 

First we prove analogs of our results in $\mathbb{R}^n$ with restrictions:

\begin{thm}[Linear Configurations] Let $G$ be a connected Lie group and let $\sigma: P(\mathfrak{g}) \to G$ be a linear unoriented configuration.
Then $e \in |\sigma|$. A similar result holds for linear oriented configurations as long as the dimension of $G$ is odd. 
\label{thm: linear case}
\end{thm}
\begin{proof}
Let $\mu: P(\mathfrak{g}) \to \mathfrak{g}$ denote a continuous lift of $\sigma$. By the proof of Theorem~\ref{thm: Rn-unoriented}, we find that there must exist 
a line $L_0\in P(\mathfrak{g})$ such that $\mu(L_0) \in L_0$. Thus $\sigma(L_0) \in \exp(L_0)$. 
This in turn implies $\exp(L_0) \subseteq \cup_{L \in P(\mathfrak{g})} (\sigma(L) \star \exp(L))=|\sigma|$.  As $e \in \exp(L_0)$, we are done. 
The proof of the oriented case proceeds similarly using the proof of Theorem~\ref{thm: Rn-oriented} which imposes the restriction that the dimension of 
$\mathfrak{g}$ and hence of $G$ is odd.
\end{proof}

The problem now is that the translation trick that worked in $(\mathbb{R}^n,+)$ does not necessarily work in a general Lie group as left translating (in $G$) a 
linear ("liftable") configuration need not yield a configuration which is also linear. One could left translate in the Lie algebra $\mathfrak{g}$ but this 
does not work well with respect to the exponential correspondence in general. 
For example
$$
e^{\mathbb{A} + t\mathbb{B}} \neq e^{\mathbb{A}} e^{t\mathbb{B}}
$$
when the matrices $\mathbb{A}$ and $\mathbb{B}$ don't commute as one can readily check using the Baker-Campbell-Hausdorff identity.

Thus in a general connected Lie group we will settle for a partial result (though in nilpotent Lie groups we'll see we can do better!):

\begin{thm}[Taut Configurations] Let $G$ be a connected Lie group and let $\sigma: P(\mathfrak{g}) \to G$ be a taut unoriented configuration. 
Then $|\sigma|$ contains an open neighborhood of the identity and hence has positive Haar measure. A similar result holds for 
taut oriented configurations as long as the dimension of $G$ is odd.
\end{thm}
\begin{proof}
Let $\sigma$ be a taut configuration, as projective spaces (respectively spheres) are compact, the image of $\sigma$ is a compact subset 
of an open neighborhood $U$ of the identity (which is diffeomorphic to an open neighborhood of $0$ in $\mathfrak{g}$ under a local inverse of the exponential).
As multiplication $M: G \times G \to G$ is continuous, $M^{-1}(U)$ is an open neighborhood of $e \times Image(\sigma)$ in $G \times G$.
By the tube lemma (see \cite{Mu}), one has an open neighborhood $V$ of $e$ such that $V \times Image(\sigma) \subseteq M^{-1}(U)$. 
Let $W = V \cap V^{-1}$ where $V^{-1}=\{ x^{-1} | x \in V \}$ then $W$ is an open neighborhood of $e$ in $G$ such that for 
$w \in W$, with left translation operator $T_{w}$ we have $T_w \circ \sigma$ is still a continuous configuration with image in $U$ and hence is still taut.

Hence $|T_w \circ \sigma|= w \star |\sigma|$ contains $e$ by Theorem~\ref{thm: linear case}. 
Thus $|\sigma|$ contains $w^{-1}$ for all $w \in W$ and hence $W \subseteq |\sigma|$ which proves the theorem in the unoriented case.
The proof of the oriented case is similar and left to the reader.

\end{proof}

In general, in order to get stronger results in a connected Lie group, we have to look at the exponential map carefully. 
This is a well-studied subject. If $E$ is the image of the exponential map, it is a basic fact that $E$ generates the connected Lie group $G$ as a group. 
In fact $E \star E = G$ (see \cite{MoS} ). However in general, as already remarked $E \neq G$. 
It is known that $E=G$ in the case of compact connected Lie groups, connected nilpotent groups and connected solvable Lie groups of type $E$.
(see \cite{MoS}). 

Of particular use for us is a result of Dixmier and Saito (see \cite{Di} and \cite{Sa}) which states that for simply-connected solvable groups, the exponential 
map is surjective iff it is bijective iff it is a global diffeomorphism.  These in turn are shown to be equivalent to a condition that in the adjoint representation 
of the Lie algebra, no nontrivial purely imaginary roots exist (Solvable Lie algebras of type $E$). 
In particular this applies to the case of nilpotent Lie algebras. Thus in a simply-connected nilpotent Lie group, the exponential map is a diffeomorphism 
$\mathfrak{g} \to G$ just as in the case $G=(\mathbb{R}^n, +)$. Here there is then no difference between the notions of taut, linear and continuous 
configurations.
Since a connected nilpotent Lie group (or more generally a connected solvable Lie group of type $E$) is covered by a simply-connected one, and since the exponential maps commute with covering maps, 
it follows that that $exp: \mathfrak{g} \to G$ is a covering map. Basic covering space theory then 
says that a continuous Kakeya line configuration $\sigma: P(\mathfrak{g}) \to G$ has a continuous lift, i.e., is linear if and only if 
$\sigma_*$, the induced homomorphism of fundamental groups, is trivial.

The fundamental group $\pi_1(G)$ (we will omit base points as we are discussing path-connected spaces) is Abelian as Lie groups are $H$-spaces 
and furthermore the fundamental group of a nilpotent connected Lie group (or more generally a connected solvable Lie group of type $E$) acts freely on its universal cover $\mathbb{R}^n$ which implies it is torsion free. (See \cite{Br})
On the other hand, $\pi_1(P(\mathfrak{g}))$ is isomorphic to the cyclic group of order two when $n > 2$, the infinite cyclic group when $n=2$ and 
is trivial when $n \leq 1$. Thus $\sigma_*$ must be trivial as long as $n \neq 2$ as there is no nontrivial group homomorphism from a torsion group 
to a torsion-free group.

We have thus proved:

\begin{thm}
If $G$ is a connected nilpotent Lie group (or more generally a connected solvable Lie group of type $E$) of dimension $> 2$, then every continuous unoriented or oriented Kayeka configuration $\sigma$ is linear.
Furthermore $|\sigma| = G$ in the unoriented case and also in the oriented case when the dimension of $G$ is odd.
\label{thm: nilpotent}
\end{thm}
\begin{proof}
As explained in the paragraph preceding this theorem, any continuous unoriented configuration is linear (lifts) as long as we are in dimensions $> 2$.
In the case of oriented configurations, the statement is still true as spheres in real vector spaces of dimension $>2$ are simply connected. 
Thus by Theorem~\ref{thm: linear case}, we conclude $e \in |\sigma|$ for {\bf any} unoriented configuration or oriented configuration in odd dimensions.
Now we can play the translation trick as for any $T_{x}$, left translation operator, $T_{x} \circ \sigma$ is still a configuration and 
so $|T_{x} \circ \sigma| = x \star |\sigma|$ contains $e$ for all $x \in G$ and hence $|\sigma|$ contains $x^{-1}$ for all $x \in G$ from which it follows 
that $|\sigma| = G$.
\end{proof}

Theorem~\ref{thm: nilpotent} does not fully hold in dimension 2 as there exist configurations which do not lift to the Lie algebra, i.e., that are not linear. 
However it turns out that it is still true that $|\sigma|=G$ for continuous unoriented Kakeya configurations. We now seek to remove the dimension $>2$ constraint 
in this regard for theorem~\ref{thm: nilpotent}. The only nontrivial case to prove it for is the case of dimension $2$.

In dimension $2$, the only connected Lie groups up to isomorphism are $(\mathbb{R}^2,+)$, the $2$-torus $S^1 \times S^1$, cylinder $S^1 \times \mathbb{R}$ and 
the affine group $\{ \begin{bmatrix} a & b \\ 0 & 1 \end{bmatrix} | a > 0, b \in \mathbb{R} \}$ which is nonAbelian but solvable. This is because there 
are only two $2$-dimensional Lie algebras, the abelian one, which corresponds to simply connected group $(\mathbb{R}^2, +)$ and the nonAbelian one, which has simply connected Lie group equal  
to the affine group. As the affine group is centerless, it is the only connected Lie group with a nonabelian $2$-dimensional Lie algebra. 
Since the only discrete subgroups of $(\mathbb{R}^2,+)$ are free-abelian of rank $\leq 2$, it only covers cylinders and tori, and so the above list is complete.
(Any rank $2$ lattice in $\mathbb{R}^2$ is isomorphic to the standard one under an automorphism of $(\mathbb{R}^2,+)$ so all the tori and cylinders 
are isomorphic  to the standard ones.)

The exponential map is a diffeomorphism in the case of $(\mathbb{R}^2, +)$ and the affine group and so Theorem~\ref{thm: nilpotent} holds for them 
as every continuous unoriented Kakeya configuration is linear. 

If $\sigma: RP^1 \to S^1 \times S^1$ is a continuous unoriented Kakeya configuration in the Torus, the image of $\sigma_*: \pi_1(RP^1)=\mathbb{Z} \to \pi_1(S^1 \times S^1)=\mathbb{Z} \oplus \mathbb{Z}$ 
is free abelian of rank $\leq 1$. If the image of $\pi_*$ is trivial, then the same is true for any translate (they are homotopic) and both the configuration 
and its translates are linear and one can then readily prove $|\sigma|=G$ using the translation trick. Otherwise the image of $\pi_*$ is free abelian of rank 1. The corresponding 
covering group of $S^1 \times S^1$ is then a cylinder and the configuration and all its translates lift to this cylinder.

Thus to prove that $|\sigma|=G$ also holds for the torus $S^1 \times S^1$ it reduces to proving it for the cylinder $S^1 \times \mathbb{R}$. 

Thus to extend the part of Theorem~\ref{thm: nilpotent} which states that $|\sigma|=G$ so that it holds for all connected solvable Lie groups of type $E$, it remains only to prove it holds 
in the case $G=S^1 \times \mathbb{R} \cong \mathbb{C}^*$.  

Let $\sigma: RP^1 \to \mathbb{C}^*$ be continuous. We can and will view $\sigma$ as a continuous map $S^1 \to \mathbb{C}^*$ such that 
$\sigma(u)=\sigma(-u)$, i.e., we will consider the oriented Kakeya configuration which goes thru the same collection of lines twice, assigning opposite 
orientations on each run through. The underlying set of this oriented configuration is the same as its unoriented counterpart. Note $exp: \mathbb{C} \to \mathbb{C}^*$, the complex exponential, is the exponential map 
for the cylinder Lie group $\mathbb{C}^*$ and it is a covering map with kernel $K=\{ 2 \pi i n | n \in \mathbb{Z} \}$. Covering space 
theory says we can lift $\sigma$ to a path $p: [0, 2\pi] \to \mathbb{C}$, where at $p(t)$, is attached a parallel copy $L_t$ of a line of angle $t$ radians with 
respect to the $x$-axis, oriented outward from the origin. (As the Lie group $\mathbb{C}^*$ is abelian, the exponential map takes all lines, even those not through the origin, in $\mathbb{C}$ to cosets of one-parameter subgroups of 
$\mathbb{C}^*$.) To show the underlying set of any continuous unoriented Kakeya configuration in $\mathbb{C}^*$ contains the identity element $1$, it then is sufficient 
to show that $|p| = \cup_{t \in [0, 2\pi]} L_t$ contains some point of $K$. If we can show this, then every continuous unoriented Kakeya configuration in the cylinder $\mathbb{C}^*$ will contain the identity element $1$ and we can use the translation trick to show $|\sigma|=\mathbb{C^*}$ for all of them.
As mentioned before, this will complete the proof  that $|\sigma|=G$ for connected Lie groups in dimension 2. 

Thus let $p: [0, 2\pi] \to \mathbb{C}$ be a continuous map such that $L_t$ is a line that makes angle $t$ radians with the $x$-axis, going through the point $p(t)$.
We need to show $|p| = \cup_{t \in [0, 2\pi]} L_t$ contains some point of $K$.  Suppose this were not true, i.e., $|P| \cap K = \emptyset$. 

Define $H: [0, 2\pi] \times \mathbb{R} \to \mathbb{C}-K$ via $H(t,s) = p(t) + se^{it}$. Notice for fixed $t$, and varying $s$, $H(t,s)$ sweeps out the line $L_t$. 
Also notice that $Image(H)$ is exactly $|p|$ and that $H$ is continuous. Now notice that on each line of fixed $t$, the map $H$ is proper, 
thus it extends to a continuous map between one-point compactifications of $\mathbb{R}$ and $\mathbb{C}$ and gives continuous  
$\bar{H}: [0, 2 \pi] \times S^1 \to S^2-K$ where $S^2$ is the Riemann sphere.  Thus $\bar{H}$ gives a base point preserving homotopy 
of maps $h_t: S^1 \to S^2-K$ where the base points are the points at infinity. Notice as $h_{\frac{\pi}{2}}$ corresponds to a vertical line which is disjoint from 
$K$, it must be a line $x=c$ where $c \neq 0$, on the other hand $h_{0}$ corresponds to a horizontal line which is disjoint from $K$ and so 
it must be a line $y=d$ where $d \neq 2 \pi n$ for any integer $n$. Let $a$ and $b$ be points in $K$ that are below and above $d$ respectively.

We can then view $\bar{H}$ as a continuous base point preserving homotopy $[0, 2 \pi] \times S^1 \to S^2-\{a,b\}$. $S^2-\{a,b\}$ is homeomorphic to 
the punctured plane, $\mathbb{R}^2-\{ 0 \}$, under stereographic projection from the point $a \in S^2$, composed with a translation that ensures 
the point $b$ maps to $0$. 
Under this modified stereographic projection $\Psi$, the map $h_{\frac{\pi}{2}}$ clearly projects to a null-homotopic curve $S^1 \to \mathbb{R}^2-\{0\}$ while the map $h_{0}: S^1 \to \mathbb{R}^2-\{0\}$ projects to a curve with nonzero winding number about the origin. As these are homotopic via $\Psi \circ \bar{H}$, we 
achieve a contradiction. 

Thus no continuous line configuration $p: [0, 2\pi] \to \mathbb{C}$ can have $|p|$ disjoint from $K$ and so no continuous unoriented Kakeya configuration in $\mathbb{C}^*$ 
can have $|\sigma|$ not contain the identity element $1$. Thus every continuous unoriented Kakeya configuration in the cylinder, has underlying 
space which contains the identity. We are hence done proving  $|\sigma|=G$ in dimension 2 as mentioned earlier.

Thus we have proved:

\begin{cor} 
Let $G$ be a connected nilpotent Lie group (or more generally a connected solvable Lie group of type $E$) and let 
$\sigma$ be a continuous unoriented Kakeya configuration. Then $|\sigma|=G$. The same result holds in the 
oriented case when the dimension of $G$ is odd.
\end{cor}

\section{Continuous Kakeya Line Configurations in homogeneous spaces}
\label{section: homogeneous}

Let $G$ be a connected Lie group and $K$ a closed subgroup. Then it is well known that the projection map 
$\pi: G \to K \backslash G$ gives a fiber bundle with fiber $K$ and that the homogeneous space $K \backslash G$ is a smooth manifold. (See \cite{Wa}). 
In fact, this fiber bundle is a principal $K$-bundle. The translates of one-parameter subgroups in $G$ map to a distinguished class of curves in 
$K \backslash G$. (These can be constant curve if the one-parameter subgroup lies in $K$ for example). 

\begin{defn}
A continuous unoriented $G$-Kakeya configuration in $X=K \backslash G$ is a continuous map 
$\sigma: P(\mathfrak{g}) \to X$. We define $|\sigma| = \cup_{L \in \mathfrak{g}} (\sigma(L) \star \exp(L))$ where $\star$ denotes the right action of $G$ on $X$. We say the configuration lifts to a configuration in $G$ if there is a continuous 
unoriented Kakeya configuration in $G$, $\hat{\sigma}: P(\mathfrak{g}) \to G$ such that $\pi \circ \hat{\sigma} = \sigma$ where $\pi: G \to K \backslash G$ is the 
quotient map.  
Notice in this case that $|\sigma| = \pi(|\hat{\sigma}|)$. 
We make similar definitions for oriented Kakeya configurations where $P(\mathfrak{g})$ is replaced with $S(\mathfrak{g})$ 
for some choice of positive definite inner product on $\mathfrak{g}$.
\end{defn}

Notice that as fiber bundles are Serre fibrations, the question of when a continuous map like $\sigma$ lifts is a homotopy question, i.e., if $\sigma_1$ is homotopic 
to $\sigma_2$ then $\sigma_1$ has a continuous lift if and only if $\sigma_2$ does. Thus in asking when a $G$-Kakeya configuration 
$\sigma$ in $X$ lifts to one in $G$, only the unbased homotopy class of $\sigma$ is relevant.

The following lemma follows immediately from this observation:

\begin{lem}
Let $\sigma: P(\mathfrak{g}) \to X=K \backslash G$ be a continuous unoriented $G$-Kakeya configuration in $X$ whose 
image set lies in a contractible subspace of $X$ (like a chart). Then $\sigma$ is homotopic to a constant map and hence has a continuous lift to a continuous unoriented 
Kakeya configuration in $G$. Similar statements hold for oriented configurations.
\end{lem}

In general a continuous oriented $G$-Kakeya configuration $\sigma: S(\mathfrak{g}) \to X=K \backslash G$ determines a 
unbased homotopy class in $[S^{n-1}, X]$ where $n$ is the dimension of $G$. As $X$ is path connected, this in turn determines a 
$\pi_1(X)$-orbit $C$ in $\pi_{n-1}(X)$ under the action of $\pi_1(X)$ on $\pi_*(X)$. The configuration lifts to a continuous oriented Kakeya configuration 
in $G$ if and only if some element in $C$ is in the image of $\pi_*: \pi_{n-1}(G) \to \pi_{n-1}(X)$ if and only if some element in $C$ is in the kernel 
of the boundary operator $\partial: \pi_{n-1}(X) \to \pi_{n-2}(K)$.  Using these observations and similar variants for projective space, the liftability 
question can be resolved in most examples. 

Once the configuration is lifted to the Lie group, all of our previous results apply. As an example we record the following corollary:

\begin{cor}
Let $G$ be a connected solvable Lie group of type $E$ and $K$ a closed subgroup with corresponding homogeneous space $X=K \backslash G$. 
If $\sigma$ is an unoriented continuous $G$-Kakeya configuration in $X$ which lifts to one in $G$, then $|\sigma| = X$. 
The same conclusions hold for oriented configurations when the dimension of $G$ is odd.
\end{cor}

Let us look at one simple example to illustrate working with Kakeya configurations in general homogeneous spaces. 
Let $G=SO(3)$, then $G$ is a compact connected Lie group which is homeomorphic to $\mathbb{R}P^3$ as a space. 
The Lie algebra $\mathfrak{so}(3)$ is isomorphic to the Lie algebra given by $\mathbb{R}^3$ with the cross-product as bracket. 
Given a nonzero vector $\hat{v}$ in the Lie algebra, the exponential flow generates the rotation about the axis given by $\hat{v}$ in the counter 
clockwise direction (right hand rule). Thus the one-parameter subgroup of $SO(3)$ corresponding to any line in $\mathbb{R}^3$ is just 
the set of rotations about that line. 

Now consider the homogeneous space $SO(2) \backslash SO(3)$ which is diffeomorphic to the standard sphere $S^2 \subseteq \mathbb{R}^3$.
Under the quotient map, the translates of one-parameter subgroups map to closed geodesics (great circles or constant curves). 
A continuous oriented $SO(3)$-Kakeya configuration in $S^2$ is then a map $\sigma: S(\mathbb{R}^3) \to S^2$ which represents a continuously 
varying family of these curves in $S^2$ which cover all the "directions" in $SO(3)$. (Thus there can be redundancy in the directions in $S^2$).\\

(To avoid such redundancy in general, a choice of identification of the tangent space of the homogeneous space with a subspace of the tangent 
space of the Lie group $G$ can be chosen and only the corresponding subspace of the projective space of $\mathfrak{g}$ used. However 
this requires a choice of horizontal lift in the fiber bundle, i.e., a choice of connection for the principal $K$-bundle. We will not pursue this topic here.)

Since $\pi_2(SO(3))=0$, an oriented $SO(3)$-Kakeya configuration in $S^2$, $\sigma: S(\mathbb{R}^3)=S^2 \to S^2$ lifts to an oriented Kakeya configuration in $SO(3)$ if and only if it has degree $0$ as a map, i.e., is homotopic to a constant.

\section{Acknowledgements}
We would like to thank Michael Gage and Alex Iosevich for many useful discussions.

\bigskip

\noindent
Jonathan Pakianathan (corresponding author) \\
Dept. of Mathematics \\
University of Rochester, \\
Rochester, NY 14627 U.S.A. \\
E-mail address: jonpak@math.rochester.edu \\

\bigskip

\noindent
Brendan Murphy \\
Dept. of Mathematics \\
University of Rochester, \\
Rochester, NY 14627 U.S.A. \\
E-mail address: murphy@math.rochester.edu \\


\begin{thebibliography}{EMG}

\bibitem[Be1]{Be1}{Besicovitch, A.S.,}
\textit{On Kakeya's Problem and a  Similar One,}
Mathematische Zeitschrift, {\bf 27} (1928), 312-320.

\bibitem[Be2]{Be2}{Besicovitch, A.S.,}
\textit{The Kakeya Problem,}
American Mathematical Monthly, {\bf 70} (1963), 697-706.

\bibitem[Bo]{Bo}{Bourgain, J.}
\textit{Harmonic Analysis and Combinatorics,}
Mathematics: Frontiers and Perspectives, V. Arnold, M. Atiyah, P. Lax, B. Mazur eds., AMS 2000.

\bibitem[Bre]{Bre}{Bredon, G.}
\textit{Topology and Geometry,}
Springer Verlag GTM 139, New York-Heidelberg-Berlin, 1997.

\bibitem[Br]{Br}{Brown, K.}
\textit{ Cohomology of Groups,}
Springer Verlag GTM 87, New York-Heidelberg-Berlin, 1994.

\bibitem[Cu]{Cu}{Cunningham Jr., F.}
\textit{The Kakeya Problem for Simply Connected and for Star-Shaped Sets,}
American Mathematical Monthly, {\bf 78} (1971), 114-129.

\bibitem[Da]{Da}{Davies, R.}
\textit{Some remarks on the Kakeya problem,}
Proc. Cambridge Philos. Soc. {\bf 69 (3)}, (1971), 417-421.

\bibitem[Di]{Di}{Dixmier, J.}Sudan
\textit{L'application exponentielle dans les groupes de Lie Resolubles,}
Bull. Math. Soc. France {\bf 85} (1957), 113-121.

\bibitem[DH]{DH}{Dummit, E., Hablicsek, M.}
\textit{Kakeya Sets over Non-Archimedean Local Rings,}
(2011) ArXiv:1112.3381.

\bibitem[Dv]{Dv}{Dvir, Z.}
\textit{On the size of Kakeya sets in finite fields,}
Journal of American Math. Soc., {\bf 22, no. 4} (2009), 1093-1097 

\bibitem[Fe]{Fe}{Fefferman, C.}
\textit{The multiplier problem for the Ball,}
The Annals of Mathematics, 2nd Ser., Vol. {\bf 94, No. 2} (1971), 330-336.

\bibitem[Ha]{Ha}{Hatcher, A.}
\textit{Algebraic Topology,}
Cambridge University Press, 2002.

\bibitem[MS]{MS}{Milnor, J., Stasheff, J.}
\textit{Characteristic Classes,}
Princeton University Press, 1974.

\bibitem[MoS]{MoS}{Moskowitz, M, Sacksteder, R.}
\textit{Exponential Map and Differential Equations on Real Lie Groups,}
Journal of Lie Theory, Vol {\bf 13}, (2003), 291-306.

\bibitem[Mu2]{Mu2}{Munkres, J.}
\textit{Elements of Algebraic Topology,}
Benjamin/Cummings Publishing Company, Inc., 1984.

\bibitem[Mu]{Mu}{Munkres, J.}
\textit{Topology, 2nd edition,}
Prentice Hall, Inc., 2000.

\bibitem[Pa]{Pa}{Pal, J.}
\textit{Ueber ein elementares variationsproblem,}
Kongelige Danske Videnskabernes Selskab Math.-Fys. Medd. {\bf 2}, 1-35.

\bibitem[Sa]{Sa}{Saito, M.} 
\textit{Sur certains groupes de Lie rŽsolubles I,}
Sci. Papers College General Educ. Univ. of Tokyo {\bf 7} (1957), 1-11.

\bibitem[SS]{SS}{Saraf, S., Sudan, M.}
\textit{Improved lower bound on the size of Kakeya sets over finite fields,}
(2008) arXiv:0808.2499.

\bibitem[Ta]{Ta}{Tao, T.}
\textit{From Rotating Needles to Stability of Waves: Emerging Connections between Combinatorics, Analysis and PDE,}
Notices of the AMS, March (2001), 294-303.

\bibitem[Wa]{Wa}{Warner, F.}
\textit{Foundations of Differentiable Manifolds and Lie Groups,}
Springer Verlag GTM 94, New York-Heidelberg-Berlin, 1983.

\bibitem[Wo]{Wo}{Wolff, T.}
\textit{Recent Work Connected with the Kakeya Problem,}
Prospects in Mathematics, AMS 1999.

\bibitem[Ze]{Ze}{Zelmanov, E.}
\textit{Solution of the Restricted Burnside Problem for groups of odd exponent,}
Math. USSR-izv. {\bf 36 (1)}, (1991), 41-60.

\bibitem[Ze2]{Ze2}{Zelmanov, E.}
\textit{Solution of the Restricted Burnside Problem for 2-groups,}
Math. USSR-Sb. {\bf 72 (2)}, (1992), 543-565.


\end{thebibliography}
\end{document}